\documentclass{amsart}
\usepackage{hyperref}
\usepackage{graphicx}

\usepackage[utf8]{inputenc}\usepackage{microtype}\DisableLigatures[f]{encoding = *, family = *}

\usepackage{color}

\newtheorem{thm}{Theorem}[section]
\newtheorem{lem}[thm]{Lemma}

\theoremstyle{definition}
\newtheorem*{question}{Question}
\newtheorem{defin}[thm]{Definition}

\theoremstyle{remark}
\newtheorem*{rem}{Remark}

\newcommand{\M}{\mathcal{M}}
\DeclareMathOperator{\conv}{conv}

\begin{document}

\title{On a Helly-type question for central symmetry}
\author{Alexey~Garber}
\author{Edgardo Roldán-Pensado}

\address{School of Mathematical \& Statistical Sciences, The University of Texas Rio Grande Valley}
\email{alexeygarber@gmail.com}

\address{Centro de Ciencias Matemáticas, UNAM campus Morelia}
\email{e.roldan@im.unam.mx}

\begin{abstract}
We study a certain Helly-type question by Konrad Swanepoel. Assume that $X$ is a set of points such that every $k$-subset of $X$ is in centrally symmetric convex position, is it true that $X$ must also be in centrally symmetric convex position? It is easy to see that this is false if $k\le 5$, but it may be true for sufficiently large $k$. We investigate this question and give some partial results.
\end{abstract}

\maketitle

\begin{center}
    \it Dedicated to Imre Bárány on his 70th birthday.
\end{center}

\section{Introduction}
The classical Carathéodory theorem in dimension $2$ can be stated in the following equivalent way:
Let $X$ be a set of points in the plane, if any $4$ points from $X$ are in convex positions then $X$ is in convex position. In 2010, Konrad Swanepoel \cite{KS} asked the following Helly-type question which was inspired by this formulation of Carathéodory's theorem.

For brevity, we say that a set of points is in \emph{c.s.c. position} (short for centrally symmetric convex position) if it is contained in the boundary of a centrally symmetric convex body.

\begin{question}
  Does there exist a number $k$ such that for any planar set $X$ the following holds: If any $k$ points from $X$ are in c.s.c position, then the whole set $X$ is in c.s.c. position.
\end{question}

It is clear from Carathéodory's theorem that $X$ should be in convex position. One can also see that $k\geq 6$ since any $5$ points in convex position are in c.s.c. position. This follows from the fact that any $5$ points pass through a quadric curve. Since the points must be in convex position, the points lie on an ellipse, parabola, a branch of a hyperbola or a union of two lines and in each of these cases there is a centrally symmetric convex body containing these points on its boundary.

It is not clear that such a $k$ exists although we suspect that it does. In this short note, we prove the following two results in Sections \ref{sec:9} and \ref{sec:curve}.

\begin{thm}\label{thm:9}
  There is a set $X$ consisting of $9$ points that is not in c.s.c. position such that any $8$ of its points are in c.s.c. position. This implies that, if $k$ exists, then $k\ge 9$.
\end{thm}

\begin{thm}\label{thm:curve}
  Let $\Gamma$ be a closed curve such that any $6$ points of $\Gamma$ are in c.s.c. position, then $\Gamma$ bounds a centrally symmetric convex region.
\end{thm}

Before proving these theorems, we describe a way to decide whether a finite set $X$ is in c.s.c. position or not. For more information on Carathéodory's theorem and Helly-type theorems we recommend \cite{Eck} and \cite{Mat}.

\section{Centrally symmetric convex position}
We start with a useful definition.

\begin{defin}
  Let $X$ be a point set and $O$ be a point. The set $X_O$ denotes the reflection of $X$ with respect to $O$, i.e., $X_O=2O-X$.
  If $X\cup X_O$ is in convex position then we say that $O$ is an \emph{admissible center for $X$}, the set of all admissible centers is denoted by $\M_X$.
\end{defin}

Swanepoel's question can be reformulated in terms of admissible centers, since $X$ is in c.s.c. position if and only if $\M_X$ is non-empty.
The main goal of this section is to give a simple way of constructing $\M_X$.
We start with the simplest possible case. The description of the set of admissible center for a finite set $X$ can be obtained from the following simple lemmas.

\begin{lem}
  Let $\triangle=\{a,b,c\}$ be three non-collinear points. The three lines passing through the midpoints of the sides of $\conv(\triangle)$ divide the plane into $7$ regions. The set $\M_\triangle$, shown in Figure \ref{fig:triang}, is the union of the closed components of this division that do not intersect $\triangle$.
\end{lem}

The set $\M_\triangle$ is naturally represented as the union of $4$ convex subsets. We call these subsets the \emph{center-part}, \emph{$a$-part}, \emph{$b$-part} and \emph{$c$-part} as in Figure \ref{fig:triang}.

\begin{figure}[ht]
  \includegraphics{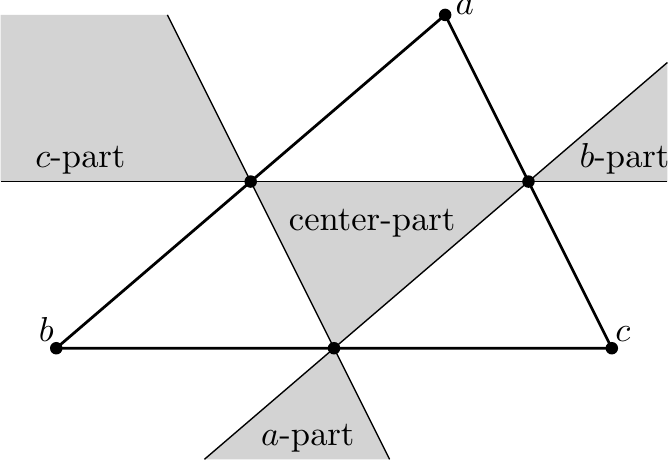}
  \caption{Set of admissible centers for a triangle.}
  \label{fig:triang}
\end{figure}

\begin{lem}\label{lem:intersection}
  For a given set $X$ in convex position we have that
  $$\M_X=\bigcap \left\{\M_Y:Y\subset X,\#(Y)=3\right\}.$$
\end{lem}

These last two lemmas provide us with a way to construct the set of admissible centers of a set with $n$ points in convex position as the intersection of $\binom n3$ sets, each of which is the union of four convex sets. We see below how we can achieve the same thing using fewer sets.

\begin{defin}
  Assume $X$ is a finite set of points in convex position such that $X$ is not contained in a line. Let $ab$ be a side of $\conv(X)$ and let $c\in X$ be a farthest point from the line $ab$. We call the triangle $abc$ a \emph{tallest triangle of $X$ with respect to side $ab$}.
\end{defin}

The tallest triangle has appeared before, at least as source of interesting questions for mathematical Olympiads (see e.g. \cite{TT} or \cite{TT2}).

\begin{thm}\label{thm:tallest}
  If $X$ is a finite set of points in convex position, then the set of admissible centers for $X$ is the intersection of the sets of admissible centers of the tallest triangles of $X$, i.e.,
  $$\M_X=\bigcap \left\{ \M_{\{a,b,c\}}: abc \text{ is a tallest triangle of } X \right\}.$$
\end{thm}
\begin{proof}
  The set $\M_X$ is included in the intersection on the right-hand side of the formula, so we only need to prove that any point from the intersection is in $\M_X$.
  
  Let $O$ be any point from the intersection and let $a$ be any point from $X$. We will show that it is possible to find a supporting line of $\conv(X\cup X_O)$ at $a$.
  
  Let $b$ be one of the neighbors of $a$ on the boundary of $\conv(X)$, say in the counter-clockwise direction. Let $abc$ be a tallest triangle of $X$ with respect to $ab$. If $O$ lies in $a$-part or $b$-part of the admissible set for triangle $abc$, then $ab$ is a supporting line for $\conv(X\cup X_O)$. Therefore $O$ lies in the $c$-part or in the central part of $\M_{\{a,b,c\}}$.
  
  Similarly, if $d$ is the other neighbor of $a$ on the boundary of $X$ (in the clockwise direction), and $ade$ is a tallest triangle of $X$ with respect to $ad$, then $O$ lies in the central part or in the $e$-part of $\M_{\{a,d,e\}}$, otherwise we are done.
 
  There are two possibilities for the positions of $c$ and $e$. Either they coincide, or $e$ is in the clockwise direction from $c$. In the case $c=e$, the only admissible point from $\M_{\{a,b,c\}}$ is the midpoint of $ac$, which also belongs to the $b$-part of $\M_{\{a,b,c\}}$. So, the line $ab$ is a supporting line of $\conv(X\cup X_O)$ as we have shown before.
  
 \begin{figure}[ht]
   \includegraphics{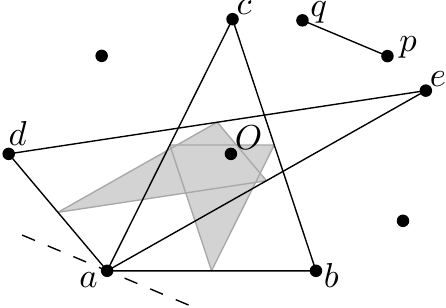}
   \caption{A supporting line of $\conv(X\cup X_O)$ at $a$.}
   \label{fig:tall}
 \end{figure}
 
  In the latter case, as shown in Figure \ref{fig:tall}, the point $O\in abc\cap ade$, and $c$ and $e$ are connected by a sequence of sides of $X$. Then there is a side $pq$ of $X$ in the angle $\angle cae$ such that $O$ is inside triangle $apq$. It is not difficult to see that $apq$ is a tallest triangle of $X$. Since $O$ is inside $apq$ and in $\M_{\{a,p,q\}}$, it is in the central part of this set of admissible centers. It follows that the line parallel to $pq$ through $a$ is also a supporting line of $\conv(X\cup X_O)$.
\end{proof}

\section{Example showing \texorpdfstring{$k\ge 9$}{k >= 9}}\label{sec:9}

In this section we prove Theorem \ref{thm:9} by giving an explicit example of a set $X$ with $9$ points such that $\M_X=\emptyset$, but $\M_Y\neq\emptyset$ for every $Y\subset X$ with $8$ points.

\begin{proof}[Proof of Theorem \ref{thm:9}]
	Start with a regular $9$-gon with center $O$ and label its vertices as $a_1$, $b_1$, $c_1$, $a_2$, $b_2$, $c_2$, $a_3$, $b_3$, $c_3$ in counter-clockwise order. Now, take the triangle $a_1a_2a_3$ and, with center $O$, scale it down by a factor of $0.93$. Then we are left with an almost regular $9$-gon such as the one shown in Figure \ref{fig:9gon}. This will be the set $X$.

\begin{figure}[ht]
  \includegraphics{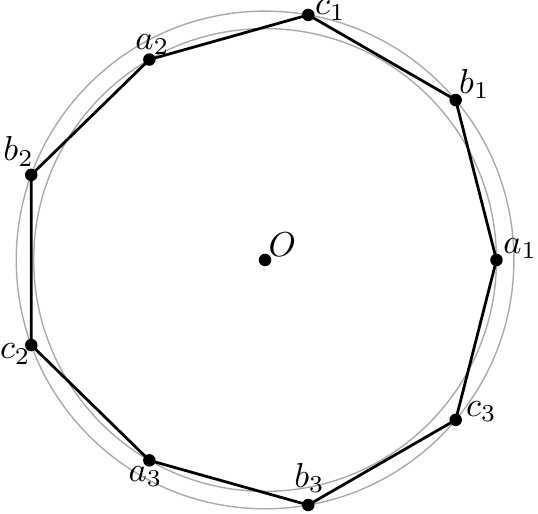}
  \caption{The $9$-gon for Theorem \ref{thm:9}.}
  \label{fig:9gon}
\end{figure}

A subset $Y$ of $X$ with $8$ points can be of two types, depending on whether or not it is missing a point $a_i$ from $X$. For each of these, a point of $\M_Y$ close to $O$ will serve as an admissible center. If we choose coordinates so that $O=(0,0)$ and $b_1=(1,0)$, then points in $\M_Y$ corresponding to $Y=X\setminus\{a_1\}$ and $Y=X\setminus\{b_2\}$ are $(0.04,0)$ and $(0.02,0)$, respectively (see Figure \ref{fig:9gon2}).

\begin{figure}[ht]
  \includegraphics{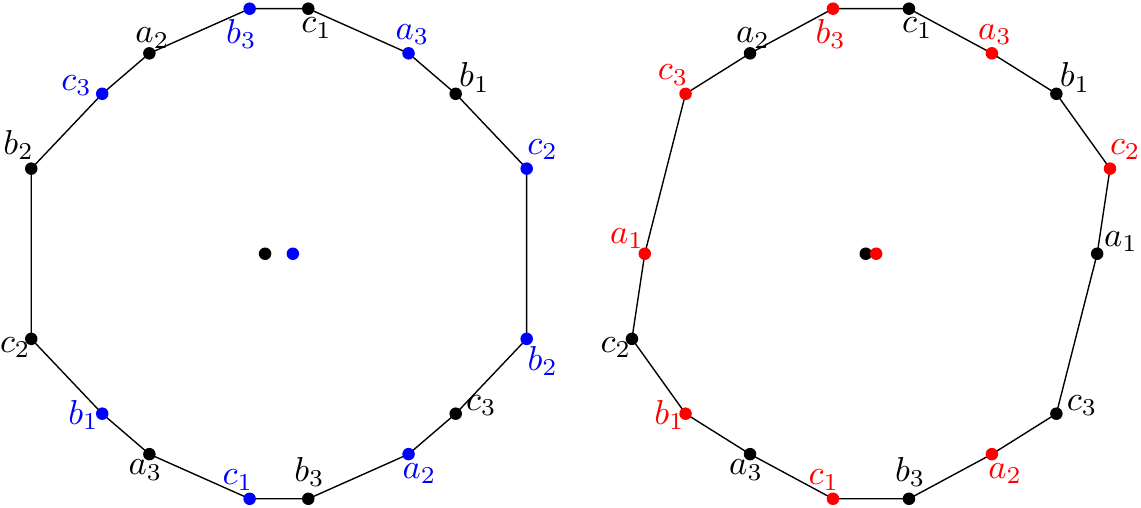}
  \caption{The original and reflected $8$-gons with their respective centers.}
  \label{fig:9gon2}
\end{figure}

All that is left is to show that $\M_X=\emptyset$. By Lemma \ref{lem:intersection}, we only need to consider the triangles determined by $X$. Let us consider first the triangle $a_1b_2c_2$, it is not hard to see that $\M_X$ must be a subset of the center part of $\M_{\{a_1,b_2,c_2\}}$. By the threefold symmetry of $X$, the same is true for the triangles $a_2b_3c_3$ and $a_3b_1c_1$. However, the center parts of these sets are triangles that do not intersect, so $\M_X$ must be empty.
\end{proof}

\section{The case of convex curves}\label{sec:curve}

In this section we show that for a convex curve $\Gamma$ the answer for Swanepoel's question is the least possible, i.e. $k=6$. For the remaining part of the section we assume that every 6 points of $\Gamma$ are in c.s.c. 

The proof of Theorem \ref{thm:curve} is based on the following simple fact, which can be proved easily using Lemma \ref{lem:intersection}.

\begin{lem}\label{lem:parallelogram}
  The set of admissible centers for the vertex-set of a parallelogram $P$ is the union the two lines passing through the center of $P$ and each parallel to a side of $P$.
\end{lem}

First we establish a few facts for $\Gamma$. Since $\Gamma$ is convex, then every point $x$ of $\Gamma$ has correctly defined one-sided tangents which are the best linear approximations of $\Gamma$ at $x$ in clockwise and counter-clockwise directions. If these lines coincide, then $\Gamma$ has a tangent at $x$ and we will call $x$ a {\it smooth point} of $\Gamma$. Due to the convexity of $\Gamma$, it may contain at most countably many non-smooth points.

\begin{lem}\label{lem:length}
If $\ell$ and $\ell'$ are two parallel supporting lines of $\Gamma$, then the lengths of the segments $\ell\cap\Gamma$ and $\ell'\cap\Gamma$ are equal.
\end{lem}
\begin{rem}
We say that a point is a segment of length zero.
\end{rem}
\begin{proof}
Suppose that the length of $\ell\cap\Gamma$ is strictly greater than the length of $\ell'\cap\Gamma$. We choose six points $a,b,c,d,e,f\in\Gamma$ in counter-clockwise order such that $a$ and $c$ are the endpoints of $\ell\cap\Gamma$, $b$ is the midpoint of $\ell\cap\Gamma$, $e$ is a point on $\ell'\cap\Gamma$, and $df$ is a segment parallel to $\ell$ such that its length is strictly between lengths of $\ell\cap\Gamma$ and $\ell'\cap\Gamma$, see Figure \ref{fig:parsupport}.
\begin{figure}[ht]
  \includegraphics{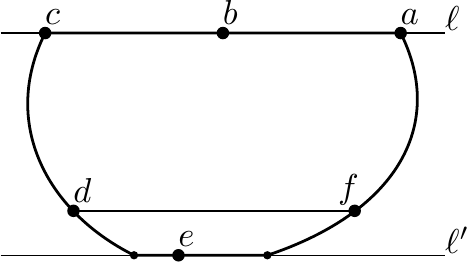}
  \caption{Parallel supporting lines cannot intersect $\Gamma$ at segments of different lengths.}
  \label{fig:parsupport}
\end{figure}

It is easy to see that these 6 points are not in c.s.c. which is a contradiction. Therefore the intersections $\ell\cap\Gamma$ and $\ell'\cap\Gamma$ have equal length.
\end{proof}

\begin{lem}\label{lem:inscribed-para}
Let $a$ be a smooth point of $\Gamma$ with tangent $\ell$, and let $b,c,d\in\Gamma$ be points such that $abcd$ is a parallelogram with sides not parallel to $\ell$. Then the line through $c$ parallel to $\ell$ supports $\Gamma$.
\end{lem}

\begin{proof}
Let $\ell'$ be the line parallel to $\ell$ through $c$. Suppose $\ell'$ doesn't support $\Gamma$. We may assume that points $a,b,c,d$ determine a counter-clockwise orientation of $\Gamma$ and $\ell'$ intersects the arc $bc$ of $\Gamma$, see Figure \ref{fig:tangent} for more details. Let $\ell_+$ be the tangent of the arc $bc$ of $\Gamma$ at $c$ and let $\ell_+'$ be the line parallel to $\ell_+$ through $a$ (see Figure \ref{fig:tangent}).

\begin{figure}[ht]
  \includegraphics{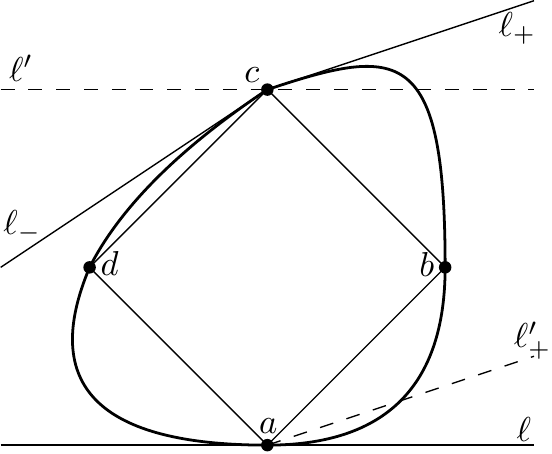}
  \caption{A curve with an inscribed parallelogram at a smooth point $a$.}
  \label{fig:tangent}
\end{figure}

We can choose a point $x$ on the arc $ab$ that is closer to $\ell$ than to $\ell_+'$, indeed, the line $\ell$ is the best linear approximation of $\Gamma$ at $a$, so each point of $\Gamma$ in a small neighborhood of $a$ is closer to $\ell$ than to $\ell_+'$. Similarly we can find a point $y$ on the arc $bc$ closer to $\ell_+$ than to $\ell'$.

For the set $X=\{a,b,c,d,x,y\}$ the set of admissible centers may consist only of the center $O$ of the parallelogram $abcd$. Indeed, according to the Lemma \ref{lem:parallelogram} $\M_X$ is contained in the union of two lines through $O$ parallel to $ab$ and $ad$. The point $x$ does not allow us to take any point in the line parallel to $ab$ except $O$ as an admissible center, and $y$ does not allow to take any point in the line parallel to $ad$ except $O$. If $y'$ is the reflection of $y$ with respect to $O$, then $a$ is strictly inside the triangle $cxy'$ as shown in Figure \ref{fig:points}. Thus we have found 6 points of $\Gamma$ which are not in c.s.c. position which is a contradiction.
\end{proof}

\begin{figure}[ht]
  \includegraphics{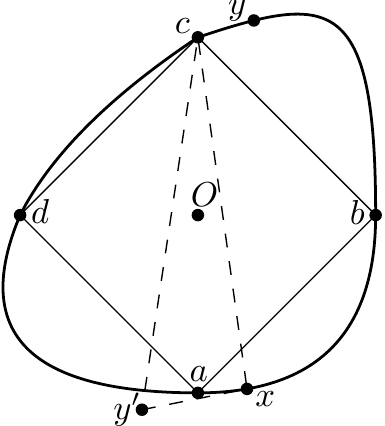}
  \caption{Six points in $\Gamma$ that are not in c.s.c position.}
  \label{fig:points}
\end{figure}

Now we proceed to the proof of the main result of this section.

\begin{proof}[Proof of Theorem \ref{thm:curve}]
The convexity of $\Gamma$ is trivial. There are two cases possible, either each supporting line of $\Gamma$ intersects $\Gamma$ at exactly one point (case 1), or there is a supporting line of $\Gamma$ that intersects $\Gamma$ in a segment of non-zero length (case 2).

\noindent{\bf Case 1.} Let $a$ be any smooth point of $\Gamma$, and let $\ell$ be the tangent of $\Gamma$ at $a$. Let $a'$ be the other point of $\Gamma$ with supporting line parallel to $\ell$.

Let $b$ be any point of $\Gamma$ other than $a$ and $a'$. The segment $ab$ is not an affine diameter of $\Gamma$, therefore there are points $c,d\in\Gamma$ such that $abcd$ is a parallelogram. From Lemma \ref{lem:inscribed-para} we get that the line through $c$ parallel to $\ell$ supports $\Gamma$ and therefore $c=a'$. Thus the central symmetry with the center at the midpoint of $aa'$ takes $b$ to another point of $\Gamma$ (the point $d$), and $\Gamma$ is centrally symmetric.

\noindent{\bf Case 2.} Let $ab$ be the intersection of $\Gamma$ with a support line $\ell$. From Lemma \ref{lem:length} we know that the other supporting line $\ell'$ of $\Gamma$ parallel to $\ell$ intersects $\Gamma$ in a segment $cd$ equal in length to $ab$. We may assume that the points $a,b,c,d$ are in counter-clockwise orientation, see Figure \ref{fig:segments}. 

\begin{figure}[ht]
  \includegraphics{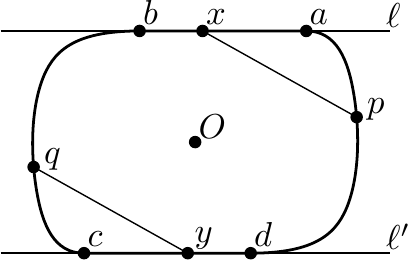}
  \caption{A curve with two equal parallel segments on the boundary.}
  \label{fig:segments}
\end{figure}

Let $p$ be a point in the interior of the arc $da$ of $\Gamma$ and $x$ be a point in the interior of the segment $ab$. Since $px$ is not an affine diameter of $\Gamma$, we can find two more points $q,y\in \Gamma$ (both depending on $x$ and $p$) such that $pxqy$ is a parallelogram. Using Lemma \ref{lem:inscribed-para} for the parallelogram $pxqy$ treating $x$ as the smooth vertex we conclude that the tangent to $\Gamma$ at $x$ is the line $\ell$, therefore the line parallel to $\ell$ through $y$ supports $\Gamma$, and $y$ belongs to the segment $cd$. Also, $q$ must be contained in the arc $bc$ of $\Gamma$. 

The center of the parallelogram $pxqy$ is equidistant from the lines $\ell$ and $\ell'$, therefore the distance from $q$ to $\ell$ is equal to the distance from $p$ to $\ell'$ and does not depend on $x$. This means that $q$ only depends on $p$ and not on $x$ and the same is true for the center $O$ of the parallelogram $pxqy$.

If for a fixed $p$ we vary $x$ in the open segment $ab$, then $y$ varies in the interior of $cd$ which has the same length as $ab$. This means that $cd$ is symmetric to $ab$ with respect to $O$ and $O$ is also the center of the parallelogram $abcd$. Thus $O$ does not depend on $p$.

Summarizing, we have shown that for every point $p$ on the arc $da$ of $\Gamma$ we can find another point $q$ of $\Gamma$ symmetric to $p$ with respect to the center of the parallelogram $abcd$. Therefore $O$ is the center of symmetry of $\Gamma$.
\end{proof}

\section{Acknowledgments}
The authors are thankful to Konrad Swanepoel for the interesting questions. We are also thankful to Imre Bárány and Jesús Jerónimo for many fruitful discussions while this work was in progress.

\end{document}